\documentclass[leqno,final]{siamart171218} 
\usepackage{amsmath}
\usepackage{graphicx,color}
\usepackage{epstopdf}
\usepackage{amssymb}
\newtheorem{remark}{Remark}

\def\O{\Omega}
\def\D{\mathcal{D}}
\def\F{\mathcal{F}}
\def\P{\mathbb{P}}
\def\E{\mathbb{E}}
\def\R{\mathbb{R}}
\def\({\left(}
\def\){\right)}
\def\div{\mbox{div}}

\def\u{\hat{u}} 
\def\w{\hat{w}}
\newcommand{\Eb}[1]{\E \left[ {#1} \right]}

\begin{document}

\title{Fully Discrete Mixed Finite Element Methods for the Stochastic Cahn-Hilliard Equation with Gradient-type Multiplicative Noise}
\markboth{XIAOBING FENG AND YUKUN LI AND YI ZHANG}{FEM for a Stochastic Cahn-Hilliard Equation}

\author{
Xiaobing Feng\thanks{Department of Mathematics, The University of Tennessee,
Knoxville, TN 37996, U.S.A. ({\tt xfeng@math.utk.edu}). The work of this author was partially supported 
by the NSF grant DMS-1620168.}
\and
Yukun Li\thanks{Department of Mathematics, The Ohio State University, Columbus, OH 43210, U.S.A. ({\tt li.7907@osu.edu}). }
\and
Yi Zhang\thanks{Department of Mathematics and Statistics, The University of North Carolina at Greensboro, Greensboro, NC 27402, U.S.A. ({\tt y\_zhang7@uncg.edu}). }
}

\maketitle

\begin{abstract}
This paper develops and analyzes some fully discrete mixed finite element methods for the stochastic Cahn-Hilliard equation  
with gradient-type multiplicative noise that is white in time and correlated in space. 
The stochastic Cahn-Hilliard equation is formally derived as a phase field formulation of 
the stochastically perturbed Hele-Shaw flow. The main result of this paper is to prove 
strong convergence with optimal rates for the proposed mixed finite element methods. 
To overcome the difficulty caused by the low regularity in time of the solution to the 
stochastic Cahn-Hilliard equation, the H\"{o}lder continuity in time with respect to various 
norms for the stochastic PDE solution is established, and it plays a crucial role in the error analysis. 
Numerical experiments are also provided to validate the theoretical results and to study the impact of
noise on the Hele-Shaw flow as well as the interplay of the geometric evolution and gradient-type noise.
\end{abstract}

\begin{keywords}
Stochastic Cahn-Hilliard equation, stochastic Hele-Shaw flow, gradient-type multiplicative noise, 
phase transition, mixed  finite element methods, strong convergence
\end{keywords}

\begin{AMS}
65N12, 
65N15, 
65N30 
\end{AMS}

\section{Introduction}\label{sec:intro} 
We consider the following stochastic Cahn-Hilliard (SCH) problem:
\begin{align}
\label{eq:SCH}
du = \Bigl[ - \Delta \Big(\epsilon \Delta u - \frac{1}{\epsilon} f(u) \Big) \Bigr] dt + \delta \nabla u  \cdot X \circ dW_t & \quad \text{in} \ \D_T:= \D \times (0,T],\\
\label{eq:SCH:b}
\frac{\partial u}{\partial n} = \frac{\partial}{\partial n} \Big(\epsilon \Delta u - \frac{1}{\epsilon} f(u)\Big) = 0 & \quad \text{in} \ \partial \D_T:=\partial \D \times (0,T], \\
\label{eq:SCH:i}
u = u_0 & \quad \text{in} \ \D \times \{ 0 \},
\end{align}
where $\D \subset \mathbb{R}^d$ ($d = 2, 3$) is a bounded domain, $n$ stands for the unit outward normal to $\partial \D$,  
and $T>0$ is a fixed number.
$W_t$ denotes a standard real-valued Wiener process on a given filtered probability space $(\O, \F, \{ \F_t: t \geq 0 \}, \P)$, 
``$\circ$" refers to the Stratonovich interpretation of the stochastic integral. $X: \mathbb{R}^d \longrightarrow \mathbb{R}^d$ 
is a smooth {\em divergence-free} vector field defined on $\D$ satisfying $X \cdot n = 0$ on $\partial \D$. 

Moreover, $f=F'$,  the derivative of a smooth double equal well potential $F$ taking its global minimum zero at $\pm 1$. 
In this paper we focus on the following quartic potential density function: 
\begin{align}\label{eq:F}
F(u) = \frac{1}{4} (u^2 - 1)^2. 
\end{align}
We note that the Stratonovich SPDE \eqref{eq:SCH} can be equivalently rewritten as the following It\^{o} SPDE: 
\begin{align}\label{eq:SCH:Ito}
du = \left[ - \Delta \Big(\epsilon \Delta u - \frac{1}{\epsilon} f(u) \Big) + \frac{\delta^2}{2} \text{div} (B \nabla u) \right] dt   + \delta \nabla u  \cdot X dW_t,
\end{align}
where $B = X \otimes X \in  \R^{d \times d}$ with $B_{ij} = X_i X_j$ ($i,j = 1, ..., d$).

By introducing the so-called chemical potential $w:= -\epsilon \Delta u + \frac{1}{\epsilon} f(u)$, 
the above primal formulation of the SCH problem can be rewritten as the following mixed formulation:
\begin{align}
\label{Msch1:mIto}
du = \Bigl[  \Delta w + \frac{\delta^2}{2} \text{div} (B \nabla u) \Bigr] dt + \delta \nabla u \cdot X d W_t & \qquad \text{in} \ \D_T, \\
\label{Msch2:mIto}
w = -\epsilon \Delta u + \frac{1}{\epsilon} f(u) & \qquad \text{in} \ \D_T,\\
\label{Msch3:mIto}
\frac{\partial u}{\partial n} = \frac{\partial w}{\partial n} = 0 & \qquad \text{on} \ \partial \D_T, \\
\label{Msch4:mIto}
u = u_0 & \qquad \text{on} \ \D \times \{ 0 \}, 
\end{align}
which will be used to develop fully discrete finite element numerical methods in this paper. 

The deterministic Cahn-Hilliard equation (i.e., $\delta = 0$) was originally introduced 
in \cite{Cahn_Hilliard58} to describe complicated phase separation and coarsening phenomena in 
a melted alloy that is quenched to a temperature at which only two different concentration phases can exist stably.  
In the equation, $u$ represents the concentration of one of two metallic components of the alloy mixture, the small parameter $\epsilon > 0$ is called the interaction length. 
Note that in \eqref{Msch1:mIto}--\eqref{Msch2:mIto}, $t$ is the {\it fast time} representing $\frac{t}{\epsilon}$ 
in the original Cahn-Hilliard formulation. The existence of bistable states suggests that nonconvex energy is associated
with the equation (cf. \cite{ABC:1994, Du_Feng19, Cahn_Hilliard58}).
The Cahn-Hilliard equation is well-known also because it closely relates to a celebrated  
moving interface problem, namely the Hele-Shaw (or Mullins-Sekerka) problem/flow. It was proved 
in \cite{P1989,ABC:1994} that, as $\epsilon \searrow 0$, the chemical potential $w := -\epsilon \Delta u + \epsilon^{-1} f(u)$ 
tends to a limit, which, together with a free boundary $\Gamma:= \cup_{0 \leq t \leq T} (\Gamma_t \times \{t\})$, 
satisfies the following Hele-Shaw (or Mullins-Sekerka) problem: 
\begin{alignat}{2}\label{HS1}
\Delta w &= 0 \qquad && \text{in} \ \D \setminus \Gamma_t, \ t \in (0,T],\\
\label{HS2}
\frac{\partial w}{\partial n} &=0 \qquad && \text{on} \ \partial \D, \ t \in (0,T],\\
\label{HS3}
w &= \sigma \kappa \qquad && \text{on} \ \Gamma_t, \ t \in (0,T], \\
\label{HS4}
V_n &= \frac{1}{2} \left[ \frac{\partial w}{\partial n} \right]_{\Gamma_t} \qquad && \text{on} \ \Gamma_t, \ t \in (0,T],\\
\label{HS5}
\Gamma_0 &= \Gamma_{00} \qquad && \text{on} \ t=0,
\end{alignat}
where $\sigma = \int_{-1}^1 \sqrt{\frac{F(s)}{2}} \, ds$,
$\kappa$ and $V_n$ are the mean curvature and the outward normal velocity of the interface $\Gamma_t$, $n$ is the unit outward 
normal to either $\partial \D$ or $\Gamma_t$, $\left[ \frac{\partial w}{\partial n} \right]_{\Gamma_t} := \frac{\partial w^+}{\partial n} - \frac{\partial w^-}{\partial n}$, and $w^+$ and $w^-$ are respectively the restriction of $w$ in the 
exterior and interior of $\Gamma_t$ in $\D$. More details about the justification of the limit can be found in 
 \cite{ABC:1994,Chen96,Stoth96} and its numerical approximations in \cite{feng2016analysis,Feng_Prohl04,Feng_Prohl05,li2017error,wu2018analysis} 
 and in \cite{Du_Feng19}. 

In applications of the Hele-Shaw flow, uncertainty may arise and come from various sources
such as thermal fluctuation, impurities of the materials and the intrinsic instabilities of the deterministic evolutions. 
Therefore, the evolution of the flow/interface under influence of noise is of great importance in applications, 
it is necessary and interesting to consider stochastic effects, and to study the impact of noise
on its phase field models and solutions, especially on their long time behaviors. This then leads to considering 
the stochastic phase field models. However, how to incorporate noises correctly into phase field models 
is often a delicate issue. 

In this paper, we consider the following stochastically perturbed Hale-Shaw flow: 
\begin{align}\label{noise:m}
V_n = \frac{1}{2} \left[ \frac{\partial w}{\partial n} \right]_{\Gamma_t} + \delta \stackrel{\circ}{W}_t X \cdot n, 
\end{align}
where a white-in-time noise multiplied by a smooth spatial coefficient function $X$ is added to the normal velocity
of the interface $\Gamma_t$, and the parameter $\delta > 0$ represents the noise intensity. 
By an heuristic argument (see \cite{RW2013,Du_Feng19} for an analogous argument), we can formally show that 
equation \eqref{eq:SCH} is a phase field formulation of the above stochastic Hele-Shaw flow.  

It should be noted that there is another stochastic Cahn-Hilliard equation, called Cahn-Hilliard-Cook (CHC) equation, which has been extensively  studied in the literature, see 
 \cite{Cook1970,PD1996-SCH,BMW2008-CHC} for PDE analysis and 
 \cite{LM2011-CHC,KLM2011-CHC,KLM2014-CHC,FKLL2018} and the references therein for its  numerical approximations. However, the noise in the 
CHC equation is additive and the parameter $\epsilon=1$ in those works. Hence, there may have no connection between  the CHC equation and the above stochastic Hele-Shaw flow.
We also note that numerical approximations of various stochastic versions of 
the following  Allen-Cahn equation: 
$$u_t=\Delta u -\frac{1}{\varepsilon^2} f(u),$$
which is a closely related to the Cahn-Hilliard equation,  
have been extensively investigated in the literature \cite{KLL2015,KKL2007,liu2017wong,KLL2018,MP2018, FLZ2019}. 
Most of those works focused on either additive noise or function-type multiplicative noise. 
Recently, finite element approximations 
of the stochastic Allen-Cahn (SAC) equation with gradient-type multiplicative 
noise had been carried out by the authors in \cite{FLZ2017}.  This
SAC equation was derived as and partially proved to be a phase field formulation of the stochastic mean curvature flow \cite{KO1982,RW2013,yip2002stochastic}. 

The goal of this paper is to extend the work of \cite{FLZ2017} to the stochastic 
Cahn-Hilliard problem \eqref{eq:SCH}--\eqref{eq:SCH:i}. Specifically, we
shall develop and analyze a fully discrete mixed finite element method for this problem, and establish strong 
convergence with rates for the proposed mixed finite element method under the assumption that the strong solution $u(\cdot, t)$ 
of the underlying SPDE problem belongs  to $W^{s,\infty}(\D)$ for almost every $t \in [0,T]$ 
and it satisfies the following high moment estimates:  
\begin{align}\label{eq:reg}
\sup_{t \in [0,T]} \Eb{\| u \|^p_{W^{s,\infty}(\D)}} \leq C_0 = C(p, \delta, \epsilon) \qquad \forall \, p \geq 1, 
\end{align}
where $\Eb{\cdot}$ denotes the expectation operator. 
It turns out that the divergence-free property of $X$ plays a key role in our analysis, which guarantees the 
sample-wise mass conservation for the strong solution to problem \eqref{eq:SCH}--\eqref{eq:SCH:i}. Another key 
ingredient for the error analysis is  the H\"{o}lder continuity estimates for the strong solution. 
To the best of our knowledge, numerical analysis has yet been done 
for the  stochastic Cahn-Hilliard equation with gradient-type multiplicative noise 
in the literature. 

The rest of the paper is organized as follows. In Section~\ref{sec:pre}, we define the weak formulation for  
problem  \eqref{eq:SCH}--\eqref{eq:SCH:i} and derive  several H\"{o}lder continuity estimates for the strong 
solution of the SPDE problem. In Section~\ref{sec:fem}, a fully discretized mixed finite element method is formulated
 and properties of the discrete inverse Laplacian operator are presented, which will be utilized to establish the well-posedness and stability of the discrete method, 
and to prove the strong convergence with rates in Section~\ref{sec:anal}. Finally, in Section~\ref{sec:numer}, we report several 
 numerical experiments to validate our theoretical results 
and to examine the interplay of the geometric parameter $\epsilon$ and the noise intensity $\delta$.  

Throughout this paper we shall use $C$ to denote a generic positive constant independent of the parameters 
$\epsilon$, $\delta$, space and time mesh sizes $h$ and $\tau$, which can take different values at different occurrences. 

\section{Preliminaries}\label{sec:pre}
Standard functional space and function notation in \cite{Adams03,BS2008} will be adopted in this paper. 
In particular,  $H^k(\D)$ for $k \geq 0$ denotes the Sobolev space of order $k$, $(\cdot, \cdot)$ and 
$\| \cdot \|_{L^2(\D)}$ denote the standard inner product and norm of $L^2(\D)$. 

In this section, we shall establish several technical lemmas about H\"{o}lder continuity estimates for the strong solution of  
problem  \eqref{eq:SCH}--\eqref{eq:SCH:i} 
that play a key role in error analysis in Section~\ref{sec:anal}. These estimates play the role of the time derivatives of 
the solution in the deterministic case.

First, we define the weak formulation for problem  \eqref{eq:SCH}--\eqref{eq:SCH:i},  based on the 
mixed formulation \eqref{Msch1:mIto}--\eqref{Msch4:mIto}, as follows: 
Seeking an $\F_{t}$-adapted and $H^1(\D)\times H^1(\D)$-valued process $(u(\cdot,t), w(\cdot,t))$ such that  
there hold $\P$-almost surely 
\begin{align}
(u(t),\phi)_{} &=  (u_0,\phi)_{} - \int_0^t (\nabla w(s), \nabla \phi)_{} \, ds - \frac{\delta^2}{2} \int_0^t (\nabla u(s) \cdot X, \nabla \phi \cdot X)_{} \, ds  \label{SCH1:mIto:w1}\\
& \qquad + \delta \int_0^t (\nabla u(s) \cdot X d W_s, \phi)_{} \qquad \forall \, \phi \in H^1(\D) \quad \forall \, t \in (0,T], \notag\\
(w(t), \varphi)_{} & = \epsilon (\nabla u(t), \nabla \varphi)_{} + \frac{1}{\epsilon} \( f(u(t)), \varphi \)_{}  \qquad \forall \, \varphi \in H^1(\D)  \quad \forall \, t \in (0,T]. \label{SCH1:mIto:w2}
\end{align}

Second, we derive a H\"{o}lder continuity estimate in time of the solution function $u$ with respect to the spatial $H^1$-seminorm.
 
\begin{lemma}\label{lem20180714_1}
Let $(u, w)$ be the solution to problem \eqref{Msch1:mIto}-\eqref{Msch4:mIto} and 
assume $u$ is sufficiently regular in the spatial variable. Then for any $t, s \in [0,T]$ with $t<s$, we have
\begin{align*}
\Eb{\|\nabla(u(s)-u(t))\|_{L^2(\D)}^2} + \epsilon \Eb{\int_t^s\|\nabla\Delta (u(\zeta)-u(t))\|_{L^2(\D)}^2d\zeta} \leq C_1(s-t),
\end{align*}
where
\begin{align*}
C_1 = C\sup_{t\leq\zeta\leq s}\Eb{\|\nabla\Delta u(\zeta)\|_{L^2(\D)}^2} + C(\delta,\frac{1}{\epsilon})\sup_{t\leq\zeta\leq s}\Eb{\| u(\zeta)\|_{H^2(\D)}^4}.
\end{align*}
\end{lemma}
\begin{proof}
Apply It\^{o}'s formula to $\Psi(u^{}(s)):=||\nabla u^{}(s)-\nabla u^{}(t)||_{L^2(\D)}^2$, and notice that
\begin{align}
\Psi^{\prime}(u^{})(v)&=2\int_{\D} (\nabla u^{}(s)-\nabla u^{}(t))\cdot \nabla v(s)dx,\label{eq20180714_1}\\
\Psi^{\prime\prime}(u^{})(m,v)&=2\int_{\D} \nabla m(s) \cdot \nabla v(s)dx,\label{eq20180714_2}
\end{align}
then we have
\begin{align}
&\|\nabla(u(s)-u(t))\|_{L^2(\D)}^2=2\int_{t}^s\bigl(\nabla(u(\zeta)-u(t)),\nabla(- \Delta (\epsilon \Delta u(\zeta) - \frac{1}{\epsilon} f(u(\zeta)) )\label{eq20180714_3}\\
&\qquad+ \frac{\delta^2}{2} \div(B\nabla u(\zeta)))\bigr)d\zeta+2\int_{t}^s\bigl(\nabla(u(\zeta)-u(t)),\nabla(\delta\nabla u(\zeta)\cdot X d{W}_{\zeta})\bigr)\notag\\
&\qquad+\delta^2\int_{t}^s(\nabla(\nabla u(\zeta)\cdot X),\nabla(\nabla u(\zeta)\cdot X))d\zeta.\notag
\end{align}

Then we obtain
\begin{align}
&\|\nabla(u(s)-u(t))\|_{L^2(\D)}^2=2\int_{t}^s\bigl(\nabla\Delta (u(\zeta)-u(t)),- \epsilon \nabla\Delta (u(\zeta) - u(t))\bigr)d\zeta\label{eq20180714_4}\\
&\qquad-2\int_{t}^s\bigl(\nabla\Delta (u(\zeta)-u(t)),\epsilon \nabla\Delta u(t)\bigr)d\zeta\notag\\
&\qquad+2\int_{t}^s\bigl(\nabla\Delta (u(\zeta)-u(t)),\frac{1}{\epsilon} \nabla f(u(\zeta))\bigr)d\zeta\notag\\
&\qquad-\delta^2\int_t^s\bigl(\Delta(u(\zeta)-u(t)),B:D^2u(\zeta)+\nabla u(\zeta)\cdot \div(B)\bigr)d\zeta\notag\\
&\qquad+2\int_{t}^s\bigl(\nabla(u(\zeta)-u(t)),\nabla(\delta\nabla u(\zeta)\cdot Xd{W}_{\zeta})\bigr)\notag\\
&\qquad+\delta^2\int_{t}^s\int_{\D}|D^2 u(\zeta)X+(\nabla X)^T\nabla u(\zeta)|^2dxd\zeta\notag,
\end{align}
where the embedding theorem from $H^2(\Omega)$ to $L^{\infty}(\Omega)$ is used in estimating the nonlinear term.

Taking the expectation on both sides of \eqref{eq20180714_4}, and using Young's inequality, the lemma is proved.
\end{proof}

It turns out we also need to control the chemical potential $w$ to handle the nonlinear terms in the error analysis. 
The following lemma establishes a H\"{o}lder continuity estimate in time for $w$ with respect to the spatial $H^1$-seminorm. 

\begin{lemma}\label{lem20180714_2}
Let $(u,w)$ be the solution to problem \eqref{Msch1:mIto}-\eqref{Msch4:mIto}
which is assumed to be sufficiently regular in the spatial variable. Then for any $t, s \in [0,T]$ with $t<s$, we have
\begin{align*}
\Eb{\|\nabla w(s)-\nabla w(t)\|_{L^2(\D)}^2} \leq C_2(s-t),
\end{align*}
where
\begin{align*}
C_2 = C\sup_{t\leq\zeta\leq s}\Eb{\|u(\zeta)\|_{H^7(\D)}^2} + C(\delta,\frac{1}{\epsilon})\sup_{t\leq\zeta\leq s}\Eb{\| u(\zeta)\|_{H^6(\D)}^6}.
\end{align*}

\end{lemma}
\begin{proof}
Define $g(u(s)) := g_1(u(s))+g_2(u(s))$, where 
\begin{align}
g_1(u(s)) &:= \|\epsilon\nabla\Delta u(s)-\epsilon\nabla\Delta u(t)\|_{L^2(\D)}^2,\notag\\
g_2(u(s)) &:= \|\frac{1}{\epsilon}\nabla f(u(s))-\frac{1}{\epsilon}\nabla f(u(t))\|_{L^2(\D)}^2.\notag
\end{align}
Notice that
\begin{align}
g_1^{\prime}(u^{})(v)&=2\epsilon^2\int_{\D}(\nabla\Delta u^{}(s)-\nabla\Delta u^{}(t)) \cdot \nabla\Delta v(s)dx,\label{eq20180714_5}\\
g_1^{\prime\prime}(u^{})(m,v)&=2\epsilon^2\int_{\D}\nabla\Delta m(s) \cdot \nabla\Delta v(s)dx,\label{eq20180714_6}
\end{align}
and
\begin{align}
g_2^{\prime}(u^{})(v)&=\frac{2}{\epsilon^2}\int_{\D}\bigl[3u^2(s)\nabla u(s)-\nabla u(s)-\nabla f(u(t))\bigr] \label{eq20180714_7} \\
&\qquad \cdot \bigl[6u(s)v(s)\nabla u(s)+3u^2(s)\nabla v(s)-\nabla v(s)\bigr]dx, \notag \\
g_2^{\prime\prime}(u^{})(m,v)&=\frac{2}{\epsilon^2}\int_{\D}\bigl[3u^2(s)\nabla u(s)-\nabla u(s)-\nabla f(u(t))\bigr] \notag \\
&\qquad \cdot  \bigl[6m(s)v(s)\nabla u(s)+6u(s)v(s)\nabla m(s)+6u(s)m(s)\nabla v(s)\bigr]dx\label{eq20180714_8} \\
&\quad+\frac{2}{\epsilon^2}\int_{\D}\bigl[6u(s)v(s)\nabla u(s)+3u^2(s)\nabla v(s)-\nabla v(s)\bigr]\notag\\
&\qquad \cdot  \bigl[3u^2(s)\nabla m(s)+6u(s)m(s)\nabla u(s)-\nabla m(s)\bigr] dx.\notag
\end{align}

Applying It\^{o}'s formula to $g(w(s)) := \|\nabla w(s)-\nabla w(t)\|_{L^2(\D)}^2$, then we have
\begin{align}\label{eq20180714_9}
&\|\nabla w(s) -\nabla w(t)\|_{L^2(\D)}^2 =2\epsilon^2\int_{t}^s\bigl(\nabla\Delta(u(\zeta)-u(t)),\nabla\Delta M_1(\zeta)\bigr)d\zeta\\
&\qquad+2\epsilon^2\int_{t}^s\bigl(\nabla\Delta(u(\zeta)-u(t)),\nabla\Delta M_2(\zeta)\bigr)d{W}_{\zeta}+\epsilon^2\int_{t}^s \int_{\D} \nabla\Delta M_2(\zeta) \notag\\
&\qquad\quad \cdot \nabla\Delta M_2(\zeta) dx d\zeta+\frac{2}{\epsilon^2}\int_{t}^s\int_{\D}\bigl[3u^2(\zeta)\nabla u(\zeta)-\nabla u(\zeta)-\nabla f(u(t))\bigr] \notag \\
&\qquad \quad \cdot  \bigl[6u(\zeta)M_1(\zeta)\nabla u(\zeta)+3u^2(\zeta)\nabla M_1(\zeta)-\nabla M_1(\zeta)\bigr]dx d\zeta\notag\\
&\qquad+\frac{2}{\epsilon^2}\int_{t}^s\int_{\D}\bigl[3u^2(\zeta)\nabla u(\zeta)-\nabla u(\zeta)-\nabla f(u(t))\bigr]\notag\\
&\qquad \quad \cdot  \bigl[6u(\zeta)M_2(\zeta) \nabla u(\zeta)+3u^2(\zeta)\nabla M_2(\zeta) -\nabla M_2(\zeta) \bigr]dxdW_{\zeta}\notag\\
&\qquad+\frac{\delta^2}{\epsilon^2}\int_{t}^s\int_{\D}\bigl[3u^2(\zeta)\nabla u(\zeta)-\nabla u(\zeta)-\nabla f(u(t))\bigr]\notag\\
&\qquad \quad \cdot  \bigl[6M_2^2(\zeta)\nabla u(\zeta)+6u(\zeta)M_2(\zeta)\nabla M_2(\zeta)+6u(\zeta)M_2(\zeta)\nabla M_2(\zeta)\bigr]dxd\zeta\notag\\
&\qquad+\frac{\delta^2}{\epsilon^2}\int_{t}^s\int_{\D}\bigl[6u(\zeta)M_2(\zeta)\nabla u(\zeta)+3u^2(\zeta)\nabla M_2(\zeta)-\nabla M_2(\zeta)\bigr]\notag\\
&\qquad \quad \cdot \bigl[3u^2(\zeta)\nabla M_2(\zeta)+6u(\zeta)M_2(\zeta)\nabla u(\zeta)-\nabla M_2(\zeta)\bigr] dx d\zeta, \notag
\end{align}
where 
\begin{align*}
M_1(\zeta)&:=- \Delta \Big(\epsilon \Delta u(\zeta) - \frac{1}{\epsilon} f(u(\zeta)) \Big) + \frac{\delta^2}{2} \div(B\nabla u(\zeta)), \\
M_2(\zeta)&:=\delta \nabla u(\zeta)\cdot X.
\end{align*}
Taking the expectation on both sides of \eqref{eq20180714_9}, 
and using Young's inequality and the embedding theorem, the lemma is proved.
\end{proof}

\section{Formulation of mixed finite element method}\label{sec:fem}
In this section we define our mixed finite element method for \eqref{SCH1:mIto:w1}-\eqref{SCH1:mIto:w2} 
and introduce several auxiliary operators that will be used in Section~\ref{sec:anal}.  

Let $t_n = n \tau$ ($n = 0, 1, ..., N$) be a uniform partition of $[0, T]$ with $\tau = T/N$ and $\mathcal{T}_h$ be a quasi-uniform triangulation of $\D$. Let $V_h$ be the finite element space given by 
\[
V_h := \{ v_h \in H^1(\D): v_h|_K \in \mathcal{P}_1(K) \quad \forall \, K \in \mathcal{T}_h \},  
\]
where $\mathcal{P}_1(K)$ denotes the space of polynomials of degree one on $K \in \mathcal{T}_h$.  
Our fully discrete mixed finite element methods for \eqref{SCH1:mIto:w1}-\eqref{SCH1:mIto:w2} is defined as seeking
$\F_{t_n}$-adapted  and $V_h \times V_h$-valued process $\{ (u_h^n, w_h^n) \}$ ($n = 1, \dots, N$) such 
that $\P$-almost surely
\begin{align}
(u^{n+1}_h,\eta_h)_{} &=  (u^n_h,\eta_h)_{} - \tau (\nabla w^{n+1}_h, \nabla \eta_h)_{} - \tau \frac{\delta^2}{2} (\nabla u^{n+1}_h \cdot X, \nabla \eta_h \cdot X)_{} \label{SCH1:m:dw1} \\
& \qquad + \delta (\nabla u^n_h \cdot X \bar{\Delta} W_{n+1}, \eta_h)_{} \qquad \forall \, \eta_h \in V_h, \notag \\
(w^{n+1}_h, v_h)_{} & = \epsilon (\nabla u^{n+1}_h, \nabla v_h)_{} + \frac{1}{\epsilon} \left( f^{n+1}, v_h \right)_{}  \qquad \forall \, v_h \in V_h, \label{SCH1:m:dw2}
\end{align}
where $\bar{\Delta}$ denotes the difference operator, $\bar{\Delta} W_{n+1} := W_{t_{n+1}} - W_{t_n} \sim \mathcal{N}(0, \tau)$ and $f^{n+1} := (u^{n+1}_h)^3 - u^{n+1}_h$. 
The initial values $(u^0_h, w^0_h)$ are chosen by solving 
\begin{align*}
(u^0_h,  v_h)&= (u_0, v_h) \qquad \forall \, v_h \in V_h, \\
(w^0_h, v_h) &= \epsilon (\nabla u^0_h, \nabla v_h) + \frac{1}{\epsilon} ((u^0_h)^3 - u^0_h, v_h) \qquad \forall \, v_h \in V_h. 
\end{align*}
Note that $u^0_h = P_h u_0$ where $P_h: L^2(\D) \longrightarrow V_h$ is the standard $L^2$-projection operator satisfying the following error estimates \cite{C1978, BS2008} 
\begin{align}\label{est:Ph}
&\| v - P_h v \|_{L^2(\D)} + h \| \nabla (v - P_h v) \|_{L^2(\D)} \leq C h^2 \|v\|_{H^2(\D)}, \\
&\| v - P_h v \|_{L^{\infty}(\D)} \leq C h^{2-d/2} \|v\|_{H^2(\D)} \label{est:Ph2} 
\end{align}
for all $v \in H^2(\D)$. 
Furthermore, a direct calculation shows that the numerical solution function $u^n_h$ satisfies
the sample-wise mass conservation property, i.e., $(u^n_h, 1) = (u_0,1) $ 
almost surely for all $n = 0, 1, ..., N$. 

Let $\mathring{V}_h$ be the subspace of  $V_h$ with zero mean, i.e., 
\begin{align} 
\mathring{V}_h := \bigl\{ v_h \in V_h:\, (v_h, 1)= 0 \bigr\}. 
\end{align}
We introduce the inverse discrete 
Laplace operator $\Delta_h^{-1}: \mathring{V}_h \rightarrow \mathring{V}_h$ as follows: given 
$\zeta \in \mathring{V}_h$, define $\Delta_h^{-1}\zeta \in \mathring{V}_h$ such that 
\begin{equation}\label{eq:distInvLap}
\bigl( \nabla(-\Delta_h^{-1}\zeta),\nabla v_h \bigr) = \bigl( \zeta, v_h \bigr) \qquad \forall\, v_h\in V_h.
\end{equation} 
For any $\zeta, \Phi \in \mathring{V}_h$, we can define the discrete $H^{-1}$ inner product by
\begin{equation}\label{eq3.3}
(\zeta,\Phi)_{-1,h}:= \bigl(\nabla(-\Delta_h^{-1}\zeta),\nabla(-\Delta_h^{-1}\Phi) \bigr)
=\bigl(\zeta,-\Delta_h^{-1}\Phi\bigr)=\bigl(-\Delta_h^{-1}\zeta,\Phi\bigr). 
\end{equation}
The induced mesh-dependent $H^{-1}$ norm is given by
\begin{equation}\label{eq3.4}
\|\zeta\|_{-1,h}:=\sqrt{(\zeta,\zeta)_{-1,h}}
=\mathop{\sup}_{\Phi\in\mathring{V}_h}\limits\frac{(\zeta,\Phi)}{|\Phi|_{H^1(\D)}}.
\end{equation}
The following properties can be easily 
verified (cf. \cite{Aristotelous12}):
\begin{alignat}{2}\label{eq3.5}
|(\zeta,\Phi)| &\leq\|\zeta\|_{-1,h} |\Phi|_{H^1(\D)} &&\qquad\forall\,\zeta\in \mathring{V}_h, \ \Phi\in \mathring{V}_h, \\
\|\zeta\|_{-1,h} &\leq C\|\zeta\|_{L^2(\D)} &&\qquad\forall\,\zeta\in \mathring{V}_h, \label{eq3.6}
\end{alignat}
and, if $\mathcal{T}_h$ is quasi-uniform, we further have 
\begin{equation}\label{eq3.7}
\|\zeta\|_{L^2(\D)}\leq C h^{-1}\|\zeta\|_{-1,h}\qquad\forall\,\zeta\in \mathring{V}_h.
\end{equation}


Setting $\u^n_h = u^n_h - \bar{u}_0$ and $\w^n_h = w^n - \bar{w}^n_h$, where $\bar{v} = |\D|^{-1} (v, 1)$, 
we can equivalently formulate \eqref{SCH1:m:dw1}--\eqref{SCH1:m:dw2} as: 
seeking $\F_{t_n}$-adapted and $\mathring{V}_h \times \mathring{V}_h$-valued process $\{ (\u_h^n, \w_h^n) \}$ ($n = 1, \dots, N$) such 
that $\P$-almost surely 
\begin{align}
(\u^{n+1}_h,\eta_h)_{} &=  (\u^n_h,\eta_h)_{} - \tau \bigl(\nabla \w^{n+1}_h, \nabla \eta_h\bigr)  - \tau \frac{\delta^2}{2} \bigl(\nabla \u^{n+1}_h \cdot X, \nabla \eta_h \cdot X \bigr) \label{SCH1:m:dw1:v2} \\
& \qquad + \delta \bigl(\nabla \u^n_h \cdot X \bar{\Delta} W_{n+1}, \eta_h \bigr)  \qquad \forall \, \eta_h \in \mathring{V}_h, \notag \\
(\w^{n+1}_h, v_h)_{} & = \epsilon \bigl(\nabla \u^{n+1}_h, \nabla v_h \bigr)  + \frac{1}{\epsilon} \bigl( \hat{f}^{n+1}, v_h \bigr)   \qquad \forall \, v_h \in \mathring{V}_h, \label{SCH1:m:dw2:v2}
\end{align}
where $\hat{f}^{n+1} = (\u^{n+1}_h + \bar{u}_0)^3 - (\u^{n+1}_h + \bar{u}_0)$.  

The next theorem establishes the well-posedness for the proposed numerical method.

\begin{theorem}\label{thm:existence} 
The scheme \eqref{SCH1:m:dw1}--\eqref{SCH1:m:dw2} (or \eqref{SCH1:m:dw1:v2}--\eqref{SCH1:m:dw2:v2}) is uniquely solvable, provided that the following mesh constraint is satisfied 
\begin{align}\label{mesh-constraint}
\tau \leq C (\epsilon^{-3} + \epsilon^{-1} \delta^4)^{-1}. 
\end{align}
\end{theorem}

\begin{proof}
For any $v_h \in \mathring{V}_h$, let $\eta_h = - \Delta^{-1}_h v_h \in \mathring{V}_h$ in \eqref{SCH1:m:dw1:v2}, we have 
\begin{align}\label{exist:1}
(\u^{n+1}_h, - \Delta^{-1}_h v_h) &= (\u^n_h, - \Delta^{-1}_h v_h) - \tau (\nabla \w^{n+1}_h, \nabla (- \Delta^{-1}_h v_h)) \\
& \qquad - \tau \frac{\delta^2}{2} \bigl(\nabla \u^{n+1}_h \cdot X, \nabla (- \Delta^{-1}_h v_h) \cdot X\bigr) \notag \\
& \qquad + \delta \bigl(\nabla \u^n_h \cdot X \bar{\Delta} W_{n+1}, - \Delta^{-1}_h v_h \bigr). \notag
\end{align}
By \eqref{eq3.3}, \eqref{SCH1:m:dw2:v2} and integration by parts, we can rewrite \eqref{exist:1} as 
\begin{align}\label{exist:2}
(\u^{n+1}_h, v_h)_{-1,h} &+ \tau \epsilon (\nabla \u^{n+1}_h, \nabla v_h) + \frac{\tau}{\epsilon} \left( (\u^{n+1}_h + \bar{u}_0)^3, v_h \right) - \frac{\tau}{\epsilon} (\u^{n+1}_h, v_h) \\
& \quad + \tau \frac{\delta^2}{2} \bigl(\nabla \u^{n+1}_h \cdot X, \nabla (- \Delta^{-1}_h v_h) \cdot X \bigr) - (\u^n_h, v_h)_{-1,h} \notag \\
& \quad - \delta \bigl(\u^n_h X, \nabla (- \Delta^{-1}_h v_h) \bigr) \bar{\Delta} W_{n+1} = 0 \qquad \forall \, v_h \in \mathring{V}_h. \notag 
\end{align}

Now we define $B: \mathring{V}_h \longrightarrow \mathring{V}_h$ by 
\begin{align}\label{exist:3}
(B(z), v_h)_{-1,h} &= (z, v_h)_{-1,h} + \tau \epsilon (\nabla z, \nabla v_h) + \frac{\tau}{\epsilon} \left( (z + \bar{u}_0)^3, v_h \right) \\
& \quad - \frac{\tau}{\epsilon} (z, v_h) + \tau \frac{\delta^2}{2} \bigl(\nabla z \cdot X, \nabla (- \Delta^{-1}_h v_h) \cdot X\bigr) - (\u^n_h, v_h)_{-1,h} \notag \\
& \quad - \delta \bigl(\u^n_h X, \nabla (- \Delta^{-1}_h v_h) \bigr) \bar{\Delta} W_{n+1} \qquad \forall \, z, v_h \in \mathring{V}_h. \notag 
\end{align}
For any $v_h \in \mathring{V}_h$, we have 
\begin{align}\label{exist:4}
(B(v_h), v_h)_{-1,h} &= \| v_h \|_{-1,h}^2 + \tau \epsilon \| \nabla v_h \|_{L^2(\D)}^2 + \frac{\tau}{\epsilon} \left( (v_h + \bar{u}_0)^3, v_h \right)  \\
& \qquad - \frac{\tau}{\epsilon} \| v_h \|^2_{L^2(\D)} + \tau \frac{\delta^2}{2} \bigl(\nabla v_h \cdot X, \nabla (- \Delta^{-1}_h v_h) \cdot X \bigr)  \notag \\
&\qquad  - (\u^n_h, v_h)_{-1,h}  - \delta \bigl(\u^n_h X, \nabla (- \Delta^{-1}_h v_h) \bigr) \bar{\Delta} W_{n+1}. \notag 
\end{align}
Notice that 
\begin{align}\label{exist:5}
\frac{\tau}{\epsilon} \left( (v_h + \bar{u}_0)^3, v_h \right) &= \frac{\tau}{\epsilon} \( v_h^3 + 3 v_h^2 \bar{u}_0 + 3 v_h \bar{u}_0^2 + \bar{u}^3_0, v_h \) \\
&= \frac{\tau}{\epsilon} \| v_h \|^4_{L^4(\D)} + \frac{3 \tau \bar{u}_0}{\epsilon} \| v_h \|^3_{L^3(\D)} + \frac{3 \tau \bar{u}_0^2}{\epsilon} \| v_h \|^2_{L^2(\D)} \notag \\ 
&\geq \frac{3 \tau \bar{u}_0^2}{4 \epsilon} \| v_h \|^2_{L^2(\D)}, \notag 
\end{align}
where we had used $$\frac{3 \tau \bar{u}_0}{\epsilon} \| v_h \|^3_{L^3(\D)} \geq - \frac{\tau}{\epsilon} \| v_h \|_{L^4(\D)}^4 - \frac{9 \tau \bar{u}_0^2}{4 \epsilon} \| v_h \|_{L^2(\D)}^2$$ to obtain the last inequality. Moreover, we have 
\begin{align}\label{exist:6}
- \frac{\tau}{\epsilon} \| v_h \|^2_{L^2(\D)} \geq - \frac{C \tau}{\epsilon^3} \| v_h \|_{-1,h}^2 - \frac{\tau \epsilon}{4} \| \nabla v_h \|_{L^2(\D)}^2 
\end{align}
by \eqref{eq3.5}, and 
\begin{align}\label{exist:7}
\tau \frac{\delta^2}{2} (\nabla v_h \cdot X, \nabla (- \Delta^{-1}_h v_h) \cdot X) \geq - \frac{\tau \epsilon}{4} \| \nabla v_h \|_{L^2(\D)}^2 - \frac{C \tau \delta^4}{\epsilon} \| v_h \|^2_{-1,h} 
\end{align}
by the Cauchy-Schwarz inequality. Combining \eqref{exist:4}--\eqref{exist:7} yields 
\begin{align}\label{exist:8}
(B(v_h), v_h)_{-1,h} &\geq \Bigl\{ \bigl[ 1 - C \tau \( \epsilon^{-3} + \delta^4 \epsilon^{-1} \) \bigr] \| v_h \|_{-1,h} - \| \u^n_h \|_{-1,h}    \\
& \qquad\quad - C \| \u^n_h \|_{L^2(\D)} \Bigr\} \| v_h \|_{-1,h} + \frac{\tau \epsilon}{2} \| \nabla v_h \|_{L^2(\D)}^2.   \notag
\end{align}
Hence we have 
\begin{align}\label{exist:9}
&(B(v_h), v_h)_{-1,h} \geq 0 \qquad \forall \, v_h \in \mathring{V}_h, \\
&\| v_h \|_{-1,h} = C (\| \u^n_h \|_{-1,h} + \| \u^n_h \|_{L^2(\D)}), \label{exist:9b}
\end{align}
provided that the mesh constraint \eqref{mesh-constraint} holds. 
It follows from Brouwer's fixed point theorem (cf. \cite{GR1986} and \cite[Lemma 7.2]{R2008}) that there exists $\u^{n+1}_h \in \mathring{V}_h$ such that 
\begin{align}\label{exist:10}
B(\u^{n+1}_h) = 0, \qquad \| \u^{n+1}_h \|_{-1,h} \leq C (\| \u^n_h \|_{-1,h} + \| \u^n_h \|_{L^2(\D)}), 
\end{align}
which also implies the existence of the solution to \eqref{exist:2}. This $\u^{n+1}_h$ together with $\w^{n+1}_h$ determined by \eqref{SCH1:m:dw2:v2} solves 
\eqref{SCH1:m:dw1:v2}--\eqref{SCH1:m:dw2:v2}. 

Next, it suffices to establish the uniqueness of the solution to \eqref{exist:2}. Assume $\u^{n+1}_{h,1}$ and $\u^{n+1}_{h,2}$ are two solutions to \eqref{exist:2}. Denote $U^{n+1}_h = \u^{n+1}_{h,1} - \u^{n+1}_{h,2}$, we have 
\begin{align}\label{unique:1}
&(U^{n+1}_h, v_h)_{-1,h} + \tau \epsilon (\nabla U^{n+1}_h, \nabla v_h)  
+ \frac{\tau}{\epsilon} \Bigl( (\u^{n+1}_{h,1} + \bar{u}_0)^3 - (\u^{n+1}_{h,2} + \bar{u}_0)^3, v_h \Bigr)  \\
&\quad  - \frac{\tau}{\epsilon} (U^{n+1}_h, v_h) + \tau \frac{\delta^2}{2} (\nabla U^{n+1}_h \cdot X, \nabla (- \Delta^{-1}_h v_h) \cdot X) = 0 \notag \qquad \forall \, v_h \in \mathring{V}_h. 
\end{align}
Taking $v_h = U^{n+1}_h$ in \eqref{unique:1}, using the fact that 
\begin{align}\label{unique:2}
\frac{\tau}{\epsilon} \Bigl( (\u^{n+1}_{h,1} + \bar{u}_0)^3 - (\u^{n+1}_{h,2} + \bar{u}_0)^3, U^{n+1}_h \Bigr) \geq 0, 
\end{align}
and the similar estimates to \eqref{exist:6} and \eqref{exist:7}, we obtain 
 \begin{align}\label{unique:3}
\left[ 1 - C \tau \( \epsilon^{-3} + \delta^4 \epsilon^{-1} \) \right] \| U^{n+1}_h \|^2_{-1,h} + \frac{\tau \epsilon}{2} \| \nabla U^{n+1}_h \|^2_{L^2(\D)} \leq 0. 
\end{align}
Therefore, under the mesh constraint \eqref{mesh-constraint}, we conclude that $U^{n+1}_h = 0$. This completes the proof. 
\end{proof}

Next theorem derives an {\em a priori} estimates for $u^n_h$. 

\begin{theorem}\label{thm:-1h-stability} 
Let $(u^n_h, w^n_h) \in V_h \times V_h$ be the unique solution of \eqref{SCH1:m:dw1}--\eqref{SCH1:m:dw2} and suppose the mesh constraint \eqref{mesh-constraint} is satisfied, there holds 
 \begin{align}\label{stability:-1h}
\sup_{0 \leq n \leq N} \Eb{\| u^n_h \|_{-1,h}^2} + \Eb{\sum_{n=1}^N \tau \| \nabla u^n_h \|_{L^2(\D)}^2} \leq C(\delta, \epsilon^{-1}). 
\end{align}
\end{theorem}

\begin{proof}
It suffices to prove the estimates for the solution to \eqref{SCH1:m:dw1:v2}--\eqref{SCH1:m:dw2:v2}. Taking $\eta_h = - \Delta^{-1}_h \u^{n+1}_h$ in \eqref{SCH1:m:dw1:v2} and $v_h = \u^{n+1}_h$ in \eqref{SCH1:m:dw2:v2}, we have 
\begin{align}\label{stab:1}
& (\u^{n+1}_h - \u^{n}_h, - \Delta^{-1}_h \u^{n+1}_h) + \tau \epsilon \| \nabla \u^{n+1}_h \|^2_{L^2(\D)} \\
& \qquad = - \tau \frac{\delta^2}{2} \bigl(\nabla \u^{n+1}_h \cdot X, \nabla (- \Delta^{-1}_h \u^{n+1}_h) \cdot X \bigr)  \notag \\
& \quad\qquad\quad  + \delta \bigl(\nabla \u^n_h \cdot X \bar{\Delta} W_{n+1}, - \Delta^{-1}_h \u^{n+1}_h\bigr) - \frac{\tau}{\epsilon} 
\bigl( \hat{f}^{n+1}, \u^{n+1}_h \bigr). \notag
\end{align} 

Notice that 
\begin{align}\label{stab:2}
\bigl(\u^{n+1}_h - \u^{n}_h, - \Delta^{-1}_h \u^{n+1}_h \bigr) = \frac12 \| \u^{n+1}_h \|^2_{-1,h} - \frac12 \| \u^{n}_h \|^2_{-1,h} + \frac12 \| \u^{n+1}_h - \u^{n}_h \|^2_{-1,h}, 
\end{align}
and the right-hand side of \eqref{stab:1} can be estimated as follows: 
\begin{align}\label{stab:3}
 - \tau \frac{\delta^2}{2} & \bigl(\nabla \u^{n+1}_h \cdot X, \nabla (- \Delta^{-1}_h \u^{n+1}_h ) \cdot X \bigr) \\
& \qquad \qquad \leq \frac{C \delta^4 \tau}{\epsilon} \| \u^{n+1}_h \|^2_{-1,h} + \frac{\epsilon \tau}{16} \| \nabla \u^{n+1}_h \|^2_{L^2(\D)}, \notag
\end{align}
and 
\begin{align}\label{stab:4}
&\delta \bigl(\nabla \u^n_h \cdot X \bar{\Delta} W_{n+1}, - \Delta^{-1}_h \u^{n+1}_h \bigr) \\
&= - \delta \bigl(\u^n_h X \bar{\Delta} W_{n+1}, \nabla (- \Delta^{-1}_h (\u^{n+1}_h - \u^{n}_h) ) \bigr)  
- \delta \bigl(\u^n_h X \bar{\Delta} W_{n+1}, \nabla (- \Delta^{-1}_h \u^n_h) \bigr) \notag \\
&\leq  C \delta^2 \| \u^n_h \|^2_{L^2(\D)} (\bar{\Delta} W_{n+1})^2 + \frac12 \| \u^{n+1}_h - \u^{n}_h \|^2_{-1,h} \notag \\
& \qquad - \delta \bigl(\u^n_h X \bar{\Delta} W_{n+1}, \nabla (- \Delta^{-1}_h \u^n_h) \bigr) \notag \\
&\leq  \frac{\epsilon}{16} \| \nabla \u^{n}_h \|^2_{L^2(\D)} (\bar{\Delta} W_{n+1})^2 + \frac{C \delta^4}{\epsilon} \| \u^{n}_h \|^2_{-1,h} (\bar{\Delta} W_{n+1})^2 + \frac12 \| \u^{n+1}_h - \u^{n}_h \|^2_{-1,h} \notag \\
& \qquad - \delta \bigl(\u^n_h X \bar{\Delta} W_{n+1}, \nabla (- \Delta^{-1}_h \u^n_h) \bigr)\notag
\end{align}
by integration by parts and \eqref{eq3.5}. Moreover, it follows from \eqref{eq3.5} that 
\begin{align}\label{stab:5}
- \frac{\tau}{\epsilon} \bigl( \hat{f}^{n+1}, \u^{n+1}_h \bigr) &\leq - \frac{\tau}{2 \epsilon} \| \u^{n+1}_h \|^4_{L^4(\D)} - \frac{3 \tau}{\epsilon} \bar{u}_0^2 \| \u^{n+1}_h \|^2_{L^2(\D)} + \frac{C \tau}{\epsilon} \| \u^{n+1}_h \|^2_{L^2(\D)} \\
& \leq \frac{\epsilon \tau}{16} \| \nabla \u^{n+1}_h \|^2_{L^2(\D)} + \frac{C \tau}{\epsilon^3} \| \u^{n+1}_h \|^2_{-1,h}.  \notag 
\end{align}

Taking the expectation on both sides of \eqref{stab:1}, summing over $n = 0, 1, ..., \ell-1$ with $1 \leq \ell \leq N$, using \eqref{stab:2}--\eqref{stab:5} and the fact that
\[
\Eb{\delta (\u^n_h X \bar{\Delta} W_{n+1}, \nabla (- \Delta^{-1}_h \u^n_h))} = 0, 
\]
we get 
\begin{align}\label{stab:6}
&\Bigl[ \frac12 - C \tau \( \epsilon^{-3} + \delta^4 \epsilon^{-1} \) \Bigr] \Eb{\| \u^{\ell}_h \|^2_{-1,h}} 
+ \frac{\epsilon}{16} \Eb{\sum_{n = 1}^\ell \tau \| \nabla \u^{n}_h \|^2_{L^2(\D)}} \\
& \leq C (\delta^4 \epsilon^{-1} + \epsilon^{-3}) \tau \sum_{n=1}^{\ell-1} \Eb{\| \u^n_h \|^2_{-1,h}} + \frac12 \Eb{\| u^0_h \|^2_{-1,h}}   + \frac{\epsilon \tau}{16} \Eb{\| \nabla \u^0_h \|^2_{L^2(\D)}}. \notag 
\end{align}

Finally, \eqref{stability:-1h} follows from applying the discrete Gronwall inequality to \eqref{stab:6}. The proof is complete.
\end{proof}

\section{Strong convergence analysis}\label{sec:anal}
The goal of this section is to establish the strong convergence with rates for the fully discrete mixed finite element 
method defined in the previous section. To the end, we introduce for $n=0,1,2,...,N$,
\begin{align*}
&E^n = u(t_n) - u_h^n := \Theta^n+\Phi^n,\\
&\Theta^n := u(t_n) - P_hu(t_n),\qquad\Phi^n := P_hu(t_n)-u_h^n,\\
&G^n = w(t_n) - w_h^n := \Lambda^n+\Psi^n,\\
&\Lambda^n := w(t_n) - P_hw(t_n),\qquad\Psi^n := P_hw(t_n)-w_h^n. 
\end{align*}

With the help of H\"{o}lder continuity estimates derived in Section~\ref{sec:pre}, we are able to prove strong convergence with rates for $E^n$, which is stated in the following theorem. 

 \begin{theorem}\label{thm:errest:1}
Under the mesh constraint \eqref{mesh-constraint}, there holds 
\begin{align}\label{errest:1}
\sup_{0 \leq n \leq N} \Eb{\| E^n \|^2_{-1,h}} + \Eb{\sum_{n = 1}^N \tau \| \nabla E^n \|^2_{L^2(\D)}} \leq C(T, \epsilon^{-1}, \delta) \left( \tau + h^2 \right). 
\end{align}
\end{theorem}

\begin{proof}
Subtracting \eqref{SCH1:m:dw1}--\eqref{SCH1:m:dw2} from \eqref{SCH1:mIto:w1}--\eqref{SCH1:mIto:w2} 
after substituting $0$ by $t_n$, and $t$ by $t_{n+1}$, we get $\mathbb{P}$-almost surely
\begin{align}
(E^{n+1},\eta_h)_{} &=  (E^n,\eta_h)_{} - \int_{t_n}^{t_{n+1}} (\nabla w(s)-\nabla w_h^{n+1}, \nabla \eta_h)_{} \, ds  \label{errq:1}\\
& \qquad - \frac{\delta^2}{2} \int_{t_n}^{t_{n+1}} \big( (\nabla u(s)-\nabla u_h^{n+1}) \cdot X, \nabla \eta_h \cdot X \big)_{} \, ds \notag\\
&\qquad+ \delta \int_{t_n}^{t_{n+1}} \big( (\nabla u(s)-\nabla u_h^n) \cdot X, \eta_h \big)_{} d W_s \qquad \forall \, \eta_h \in V_h, \notag\\
(G^{n+1}, v_h)_{} &= \epsilon (\nabla E^{n+1}, \nabla v_h)_{} + \frac{1}{\epsilon} \big( f(u(t_{n+1})) - f^{n+1}, v_h \big)_{}  \qquad \forall \, v_h \in V_h.\label{errq:2}
\end{align}

Since $\Phi^{n+1}(\omega) \in \mathring{V}_h$, setting $\eta_h = - \Delta^{-1}_h \Phi^{n+1} (\omega)$ in \eqref{errq:1} and $v_h = \tau \Phi^{n+1}(\omega)$ in \eqref{errq:2}, it follows from the definition of $\Delta^{-1}_h$ (cf. \eqref{eq:distInvLap}) that  
\begin{align}
\label{errq:3}
&(\Phi^{n+1} - \Phi^n, \Phi^{n+1})_{-1,h} = (\Theta^{n+1} - \Theta^n, \Delta^{-1}_h \Phi^{n+1}) - \tau (\Lambda^{n+1}, \Phi^{n+1}) \\
& \qquad \quad  - \tau (\Psi^{n+1}, \Phi^{n+1}) + \int_{t_n}^{t_{n+1}} (\nabla w(s) - \nabla w(t_{n+1}), \nabla \Delta^{-1}_h \Phi^{n+1}) \, ds \notag \\
& \qquad \quad + \frac{\delta^2}{2} \int_{t_n}^{t_{n+1}} (\nabla (\Theta^{n+1} + \Phi^{n+1}) \cdot X, \nabla \Delta^{-1}_h \Phi^{n+1} \cdot X) \, ds \notag \\
& \qquad \quad + \frac{\delta^2}{2} \int_{t_n}^{t_{n+1}} ( \nabla (u(s) - u(t_{n+1}) \cdot X ), \nabla \Delta^{-1}_h \Phi^{n+1} \cdot X) \, ds \notag \\
& \qquad \quad - \delta \int_{t_n}^{t_{n+1}} (\nabla (\Theta^n + \Phi^n) \cdot X, \Delta^{-1}_h \Phi^{n+1}) \, d W_s \notag \\
& \qquad \quad - \delta \int_{t_n}^{t_{n+1}} (\nabla (u(s) - u(t_n)) \cdot X, \Delta^{-1}_h \Phi^{n+1}) \, d W_s, \notag \\
& \tau (\Lambda^{n+1} + \Psi^{n+1}, \Phi^{n+1}) = \epsilon \tau (\nabla \Theta^{n+1}, \nabla \Phi^{n+1}) + \epsilon \tau (\nabla \Phi^{n+1}, \nabla \Phi^{n+1}) \label{errq:4} \\
& \qquad \quad + \tau \frac{1}{\epsilon} \big( f(u(t_{n+1})) - f^{n+1}, \Phi^{n+1} \big). \notag
\end{align}

Combining \eqref{errq:3} and \eqref{errq:4}, and taking expectation on both sides, we have 
\begin{align}\label{errq:5}
\E &\left[ (\Phi^{n+1} - \Phi^n, \Phi^{n+1})_{-1,h} \right] + \epsilon \tau \E \left[ (\nabla \Phi^{n+1}, \nabla \Phi^{n+1}) \right] \\
&\qquad  = \E \left[ (\Theta^{n+1} - \Theta^n, \Delta^{-1}_h \Phi^{n+1}) \right] \notag \\
& \qquad \qquad - \epsilon \tau \E \left[ (\nabla \Theta^{n+1}, \nabla \Phi^{n+1}) \right] - \tau \frac{1}{\epsilon} \E \left[ \big( f(u(t_{n+1})) - f^{n+1}, \Phi^{n+1} \big) \right] \notag \\
& \qquad \qquad + \E \left[ \int_{t_n}^{t_{n+1}} (\nabla w(s) - \nabla w(t_{n+1}), \nabla \Delta^{-1}_h \Phi^{n+1}) \, ds \right] \notag \\
& \qquad \qquad + \frac{\delta^2}{2} \E \left[  \int_{t_n}^{t_{n+1}} (\nabla (\Theta^{n+1} + \Phi^{n+1}) \cdot X, \nabla \Delta^{-1}_h \Phi^{n+1} \cdot X) \, ds \right] \notag \\
& \qquad \qquad + \frac{\delta^2}{2} \E \left[  \int_{t_n}^{t_{n+1}} ( \nabla (u(s) - u(t_{n+1})) \cdot X , \nabla \Delta^{-1}_h \Phi^{n+1} \cdot X) \, ds \right] \notag \\
& \qquad \qquad - \delta \E \left[ \int_{t_n}^{t_{n+1}} (\nabla (\Theta^n + \Phi^n) \cdot X, \Delta^{-1}_h \Phi^{n+1}) \, d W_s \right] \notag \\
& \qquad \qquad - \delta \E \left[ \int_{t_n}^{t_{n+1}} (\nabla (u(s) - u(t_n)) \cdot X, \Delta^{-1}_h \Phi^{n+1}) \, d W_s \right] \notag \\
& \qquad := \sum_{i = 1}^8 T_i. \notag 
\end{align}

The left-hand side of \eqref{errq:5} can be rewritten as 
\begin{align}\label{errq:6}
\E &\left[ (\Phi^{n+1} - \Phi^n, \Phi^{n+1})_{-1,h} \right] + \epsilon \tau \E \left[ (\nabla \Phi^{n+1}, \nabla \Phi^{n+1}) \right] \\
&\quad = \frac12 \left( \E \left[ \| \Phi^{n+1} \|_{-1,h}^2 \right] - \E \left[ \| \Phi^n \|_{-1,h}^2 \right] \right) + \frac12 \E \left[ \| \Phi^{n+1} - \Phi^n \|_{-1,h}^2 \right] \notag \\
&\qquad \quad + \epsilon \tau \E \left[  \| \nabla \Phi^{n+1} \|^2_{L^2(\D)} \right]. \notag 
\end{align}

Now we estimate the right-hand side of \eqref{errq:5}. Since $P_h$ is the $L^2$-projection operator, we have 
\begin{align}\label{errq:7}
T_1 = 0.
\end{align}

For the second term on the right-hand side of \eqref{errq:5}, we have by \eqref{est:Ph} that 
\begin{align}\label{errq:8}
T_2 &\leq \frac{\epsilon}{2} \tau \E \left[ \| \nabla \Theta^{n+1} \|^2_{L^2(\D)} \right] + \frac{\epsilon}{2} \tau \E \left[ \| \nabla \Phi^{n+1} \|^2_{L^2(\D)} \right] \\
& \leq \frac{\epsilon}{2} \tau h^2 \E \left[ |u(t_{n+1})|^2_{H^2(\D)} \right] + \frac{\epsilon}{2} \tau \E \left[ \| \nabla \Phi^{n+1} \|^2_{L^2(\D)} \right]. \notag 
\end{align} 

For the third term on the right-hand side of \eqref{errq:5}, we observe that 
\begin{align}\label{errq:9}
T_3 &= - \tau \frac{1}{\epsilon} \E \left[ \big( f(u(t_{n+1})) - f(P_h(u(t_{n+1}))), \Phi^{n+1} \big) \right] \\
& \qquad - \tau \frac{1}{\epsilon} \E \left[ \big( f(P_h(u(t_{n+1}))) - f^{n+1}, \Phi^{n+1} \big) \right]. \notag 
\end{align}
First of all, we have 
\begin{align}\label{errq:10}
- \tau \frac{1}{\epsilon} \E \left[ \big( f(P_h(u(t_{n+1}))) - f^{n+1}, \Phi^{n+1} \big) \right] \leq \tau \frac{1}{\epsilon} \E \left[ \| \Phi^{n+1} \|^2_{L^2(\D)} \right] 
\end{align}
by the monotonicity property of the nonlinearity. Secondly, we can estimate the first term on the right-hand 
side of \eqref{errq:9} by 
\begin{align}\label{errq:11}
- &\frac{\tau}{\epsilon} \E \left[ \bigl( f(u(t_{n+1}) - f(P_h u(t_{n+1})), \Phi^{n+1} \bigr) \right]   \\
&  = - \frac{\tau}{\epsilon} \E \left[ \bigl( \Theta^{n+1} (u(t_{n+1})^2+u(t_{n+1})P_h u(t_{n+1})+P_h u(t_{n+1})^2-1),\Phi^{n+1} \bigr) \right]  \notag\\
& \leq \frac{\tau}{4 \epsilon} \E \Bigl[ \|u(t_{n+1})^2+u(t_{n+1})P_h u(t_{n+1})+P_h u(t_{n+1})^2-1\|_{L^{\infty}(\D)}^2 \notag \\
&\qquad   \times \|\Theta^{n+1}\|^2_{L^2(\D)} \Bigr] + \E \left[ \|\Phi^{n+1}\|^2_{L^2(\D)} \right] \notag \\
&  \leq \frac{C \tau}{4 \epsilon} \Bigl( \E \left[ \|P_h u(t_{n+1})\|_{L^\infty(\D)}^{6}
+\|u(t_{n+1})\|_{L^\infty(\D)}^{6}+|D|^{3} \right] \Bigr)^{\frac{2}{3}} \notag \\
&\qquad   \times \Bigl( \E \left[\|\Theta^{n+1}\|^{6}_{L^2(\D)} \right] \Bigr)^{\frac13} 
+ \frac{\tau}{\epsilon} \E \left[ \|\Phi^{n+1}\|^2_{L^2(\D)} \right] \notag \\
& \leq \frac{C \tau}{4 \epsilon} \left( \E \left[\|\Theta^{n+1}\|^{6}_{L^2(\D)} \right] \right)^{\frac13} 
+ \frac{\tau}{\epsilon} \E \left[ \|\Phi^{n+1}\|^2_{L^2(\D)} \right], \notag
\end{align}
where we had applied \eqref{est:Ph2} and high moment bounds \eqref{eq:reg} for the SPDE solution. Combining \eqref{errq:9}--\eqref{errq:11}, \eqref{est:Ph} and \eqref{eq3.5}, we obtain 
\begin{align}\label{errq:12}
T_3 &\leq \frac{C\tau}{\epsilon^3} \E \left[ \|\Phi^{n+1}\|_{-1,h}^2 \right] + \frac{\epsilon \tau}{4} \E \left[ \| \nabla \Phi^{n+1} \|^2_{L^2(\D)} \right] + \frac{C \tau}{\epsilon} \left( \E \left[\|\Theta^{n+1}\|^{6}_{L^2(\D)} \right] \right)^{\frac13} \\
&\leq  \frac{C\tau}{\epsilon^3} \E \left[ \|\Phi^{n+1}\|_{-1,h}^2 \right] + \frac{\epsilon \tau}{4} \E \left[ \| \nabla \Phi^{n+1} \|^2_{L^2(\D)} \right]  \notag \\
&\qquad    + \frac{C \tau h^4}{\epsilon} \left( \E \left[| u(t_{n+1}) |_{H^2(\D)}^{6} \right] \right)^{\frac13}. \notag 
\end{align}

For the fourth term on the right-hand side of \eqref{errq:5}, we have 
\begin{align}\label{errq:13}
T_4 &\leq \E \left[ \int_{t_n}^{t_{n+1}} 2 \| \nabla w(s) - \nabla w(t_{n+1}) \|^2_{L^2(\D)} + \frac18 \| \nabla \Delta^{-1}_h \Phi^{n+1} \|^2_{L^2(\D)} \, ds \right] \\
&\leq C \tau^2 + \frac18 \tau \E \left[ \| \Phi^{n+1} \|_{-1,h}^2 \right] \notag 
\end{align}
by the H\"{o}lder continuity for $\nabla w$ (cf. Lemma~\ref{lem20180714_2}). Similarly, we have 
\begin{align}\label{erreq:14}
T_6 &\leq C \delta^2 \tau^2 + \frac18 \tau \Eb{\| \Phi^{n+1} \|_{-1,h}^2}
\end{align}
by the H\"{o}lder continuity for $\nabla u$ (cf. Lemma~\ref{lem20180714_1}). 

For the fifth term on the right-hand side of \eqref{errq:5}, we have 
\begin{align}\label{errq:15}
T_5 &\leq C \delta^4 \tau h^2 + \frac{\epsilon \tau}{8} \Eb{\| \nabla \Phi^{n+1} \|^2_{L^2(\D)}} + \left( \frac18 + \frac{C \delta^4}{\epsilon} \right) \tau \Eb{\| \Phi^{n+1} \|_{-1,h}^2}. 
\end{align}

For the seventh term on the right-hand side of \eqref{errq:5}, we have by the integration by parts, the martingale property, the It\^{o} isometry and \eqref{eq3.5} that 
\begin{align}\label{errq:16}
T_7 &= - \delta \E \left[ \int_{t_n}^{t_{n+1}} \big( (\Theta^n + \Phi^n) X, \nabla \Delta^{-1}_h (\Phi^{n+1} - \Phi^n)\big) \, d W_s \right] \\
&\leq C \delta^2 \tau h^4 \Eb{\| u(t_n) \|^2_{H^2(\D)}} + C \delta^2 \tau \Eb{\| \Phi^n X \|^2_{L^2(\D)}} \notag \\
&\qquad + \frac14 \Eb{\| \Phi^{n+1} - \Phi^n \|_{-1,h}^2} \notag \\
&\leq C \delta^2 \tau h^4 \Eb{\| u(t_n) \|^2_{H^2(\D)}} + \frac14 \Eb{\| \Phi^{n+1} - \Phi^n \|_{-1,h}^2} \notag \\
&\qquad + \frac{\epsilon \tau}{16} \Eb{\| \nabla \Phi^n \|^2_{L^2(\D)}} + \frac{C \delta^4 \tau}{\epsilon} \Eb{\| \Phi^n \|^2_{-1,h}}. \notag 
\end{align}

Similarly, we have 
\begin{align}\label{errq:17}
T_8 &= - \delta \E \left[ \int_{t_n}^{t_{n+1}} \big( \nabla (u(s) - u(t_n)) \cdot X, \Delta^{-1}_h (\Phi^{n+1} - \Phi^n) \big) \, d W_s \right] \\
&\leq C \delta^2 \tau^2 + \frac14 \Eb{\| \Phi^{n+1} - \Phi^n \|_{-1,h}^2},  \notag 
\end{align}
where we have used Lemma~\ref{lem20180714_1} and the following Poincar\'{e}'s inequality:
\[
\| \Delta^{-1}_h (\Phi^{n+1} - \Phi^n) \big) \|^2_{L^2(\D)} \leq C \| \nabla \Delta^{-1}_h (\Phi^{n+1} - \Phi^n) \big) \|^2_{L^2(\D)} = C \| \Phi^{n+1} - \Phi^n \|^2_{-1,h}.
\]
 
Combining \eqref{errq:5}--\eqref{errq:8} and \eqref{errq:12}--\eqref{errq:17}, summing over $n = 0, 1, ..., \ell-1$ with $1 \leq \ell \leq N$, we have 
\begin{align}\label{errq:18}
\left( \frac18 - \frac{C \tau}{\epsilon^3} - \frac{C \delta^4 \tau}{\epsilon} \right) &\Eb{\| \Phi^\ell \|^2_{-1,h}} + \frac{\epsilon}{16} \Eb{\tau \sum_{n = 1}^\ell \| \nabla \Phi^n \|^2_{L^2(\D)}} \\
&\leq \frac12 \Eb{\| \Phi^0 \|^2_{-1,h}} + \frac{\epsilon}{16} \tau \Eb{\| \nabla \Phi^0 \|^2_{L^2(\D)}} \notag \\
&\qquad + C(\epsilon^{-1}, \delta) T (\tau + h^2 + h^4) \notag \\
&\qquad + C \left( 1 + \frac{1}{\epsilon^3} + \frac{\delta^4}{\epsilon} \right) \tau \sum_{n=1}^{\ell-1} \Eb{\| \Phi^n \|^2_{-1,h}}. \notag 
\end{align}
Therefore, under the mesh constraint \eqref{mesh-constraint}, we have by the discrete Gronwall inequality that 
\begin{align}\label{errq:21}
&\Eb{\| \Phi^\ell \|^2_{-1,h}} + \Eb{\tau \sum_{n = 1}^\ell \| \nabla \Phi^n \|^2_{L^2(\D)}} \\
&\quad  \leq C \Bigl( \Eb{\| \Phi^0 \|^2_{-1,h}} + \frac{\epsilon}{16} \tau \Eb{\| \nabla \Phi^0 \|^2_{L^2(\D)}} + \tau + h^2 \Bigr)   e^{CT(1 + \epsilon^{-3} + \delta^4 \epsilon^{-1})}. \notag 
\end{align}

Finally, the estimate \eqref{errest:1} follows from \eqref{errq:21}, the triangle inequality, and the fact that $\Phi^0 = 0$.  
The proof is complete.
\end{proof}


\begin{remark}
The error estimates in Theorem~\ref{thm:errest:1} is sub-optimal with respect to the $\| \cdot \|_{-1,h}$-seminorm,
this is due to the existence of the gradient-type noise, hence, the estimate is sharp in general. 
Numerical results in Section~\ref{sec:numer} indeed confirm the sub-optimal convergence 
whenever the noise is relatively large. However, the error is optimal in the $H^1$-seminorm which is also
confirmed by the numerical experiments. 
\end{remark}

\begin{remark}
Because the discrete Gronwall inequality was employed near the end of the proof, the error estimate 
in Theorem~\ref{thm:errest:1} depends on $\frac{1}{\epsilon}$ exponentially. 
We note that a polynomial order dependence on $\frac{1}{\epsilon}$ of the errors was achieved in the deterministic 
case (cf. \cite{Feng_Prohl04,Feng_Prohl05,feng2016analysis}) by using a PDE spectrum 
estimate result, however, such a spectrum estimate is yet proved to hold in the stochastic setting. 
\end{remark}

\section{Numerical experiments}\label{sec:numer}
In this section, we report several numerical examples to check the performance of 
the proposed fully discrete mixed finite element method and numerically study the impact of noise 
on the evolution of the solution and the stochastic Hele-Shaw flow. 

We consider the SPDE \eqref{Msch1:mIto}--\eqref{Msch4:mIto} on the square domain $\D = [-1, 1]^2$ and choose 
\begin{align*}
X = \varphi(r) [x_2, -x_1]^T, \qquad 
\varphi(r) =
\begin{cases}
e^{-\frac{0.001}{0.64 - r^2}}, & \text{if } \ r < 0.8, \\
0, & \text{if } \  r \geq 0.8, 
\end{cases}
\end{align*}
where $r  = |x|$. 
It is clear that $\text{div} X = 0$ in $\D$ and $X \cdot n = 0$ on $\partial \D$. 

Let $N_h = \text{dim} V_h$ and $\{ \psi_i \}_{i=1}^{N_h}$ be the nodal basis of $V_h$. 
Denote by ${\bf u}^{n+1}$ (resp., ${\bf w}^{n+1}$) the coefficient vector of the discrete solution 
$u^{n+1}_h = \sum_{i = 0}^{N_h} u^{n+1}_i \psi_i$ (resp., $w^{n+1}_h = \sum_{i = 0}^{N_h} w^{n+1}_i \psi_i$) 
at time $t_{n+1} = (n+1) \tau$, $n=0,..., N-1$. Then \eqref{SCH1:m:dw1}--\eqref{SCH1:m:dw2} are equivalent to
\begin{align}\label{eq:numsyst1}
\left[ {{\bf M}} + \tau \frac{\delta^2}{2} {{\bf A_X}} \right] {{\bf u}^{n+1}} + \tau {\bf A} {\bf w}^{n+1} &= {\bf M} {\bf u}^{n} + \delta \bar{\Delta} {{\bf W}_{n+1}} {\bf C_X} {\bf u}^n, \\
{\bf M} {\bf w}^{n+1} - \epsilon {\bf A} {\bf u}^{n+1} &= \frac{1}{\epsilon} {\bf N} ({\bf u}^{n+1}), \label{eq:numstyst2}
\end{align}
where $\bf{M}$ and $\bf{A}$ denote respectively the mass and stiffness matrices,  $\bf{A_X}$ is the 
weighted stiffness matrix with $({\bf A_X} )_{ij} = (\nabla \psi_j \cdot X, \nabla \psi_i \cdot X)$, 
${\bf{N}} ({{\bf u}^{n+1}})$ is the nonlinear contribution corresponding to the nonlinear term 
$(f^{n+1}, v_h)$, $({\bf{C_X}})_{ij} = (\nabla \psi_j\cdot X, \psi_i)$ and $\bf{W}$ is the discrete Brownian motion 
with increments $\bar{\Delta} {{\bf W}_{n+1}} = {{\bf W}_{n+1}} - {{\bf W}_n}$. 

In all our tests, we use the Brownian motion generated by using step size $\tau_{\text{ref}} = 5 \times 10^{-5}$ 
and compute at least $M = 1000$ Monte Carlo realizations. The first test concerns a smooth initial function, aiming to verify 
the rates of convergence of the proposed method with respect to the temporal mesh size $\tau$ and the spatial mesh size $h$. 
The second and third tests are designed to investigate the influence of the noise intensity $\delta$ and the parameter 
$\epsilon$ on the stochastic evolutions for two different non-smooth initial functions. 

\subsection{Test 1}
In this test we check the rates of convergence of the method \eqref{SCH1:m:dw1}-\eqref{SCH1:m:dw2} 
with a smooth initial function 
\begin{align*}
u_0(x) = x_1^2 (1-x_1)^2 x_2^2 (1-x_2^2). 
\end{align*} 
We examine the errors $\sup_{0 \leq n \leq N} \Eb{\| E^n \|^2_{L^2(\D)}}$ and $\Eb{\sum_{n = 1}^N \tau \| \nabla E^n \|^2_{L^2(\D)}}$, 
where $E^n = u(t_n) - u^n_h$. Since the exact solution is unknown, we approximate the errors by  
\begin{align*}
&\Eb{\| E^n \|^2_{L^2(\D)}} \approx \frac{1}{M} \sum_{n=1}^M \| u^n_{h} - u^n_{\text{ref}} \|^2_{L^2(\D)},  \\
&\Eb{\| \nabla E^n \|^2_{L^2(\D)}} \approx \frac{1}{M} \sum_{n=1}^M \| \nabla (u^n_{h} - u^n_{\text{ref}}) \|^2_{L^2(\D)}. 
\end{align*}
Here $u^n_{\text{ref}}$ refers to a reference solution. 

First we examine the convergence rates in the time discretization by varying $\tau$ with the fixed parameters 
$\epsilon = 0.1$, $\delta = 5$ and $h = 2/2^6$. In Table~\ref{tab:ex1:k-rate}, we use $L^{\infty}(L^2)$ 
(resp., $L^2(H^1)$) to denote the norm (resp., seminorm) corresponding to the square root of the errors 
$\sup_{0 \leq n \leq N} \Eb{\| E^n \|^2_{L^2(\D)}}$ (resp., $\Eb{\sum_{n = 1}^N \tau \| \nabla E^n \|^2_{L^2(\D)}}$). 
We observe the half order convergence rate for the $L^2(H^1)$-error as predicted in Theorem~\ref{thm:errest:1}. 
Note that the $L^{\infty}(L^2)$ error estimate is not available in Theorem~\ref{thm:errest:1}, however it still 
provides us useful information about the accuracy of the numerical method. 

\begin{table}[htb]
\begin{center}
\begin{tabular}{|c|c|c|c|c|}
\hline 
$\tau$     & $L^{\infty}(L^2)$ error      & order  & $L^2(H^1)$ error & order \\ \hline
3.2000E-03  &  3.93E-02  &        &  1.92E-02  &       \\
1.6000E-03  &  1.52E-02  &  1.37  &  8.25E-03  &  1.22 \\
8.0000E-04  &  9.56E-03  &  0.67  &  5.01E-03  &  0.72 \\
4.0000E-04  &  6.59E-03  &  0.54  &  3.27E-03  &  0.61 \\
2.0000E-04  &  4.77E-03  &  0.47  &  2.27E-03  &  0.52 \\
1.0000E-04  &  3.28E-03  &  0.54  &  1.59E-03  &  0.52 \\
\hline 
\end{tabular}
\caption{Test 1: Temporal errors and convergence rates with $\epsilon = 0.1$, $\delta = 5$ and $h = 2/2^6$.}
\label{tab:ex1:k-rate}
\end{center}
\end{table}

Next we investigate convergence in the space discretization by varying $h$ with the fixed parameters 
$\epsilon = 0.1$, $\delta = 25$ and $\tau = 5 \times 10^{-5}$. 
Since $\delta$ is relatively large, we compute $M = 10^4$ realizations. 
From Table~\ref{tab:ex1:h-rate}, we see that the $L^2(H^1)$-error converges with order $1$ 
which is consistent with the theoretical estimate of Theorem~\ref{thm:errest:1}. Also,   the $L^{\infty}(L^2)$-error converges  
with order less than $2$, indicating a sub-optimal convergence in the lower order norm as predicted by Theorem~\ref{thm:errest:1}. 

\begin{table}[htb]
\begin{center}
\begin{tabular}{|c|c|c|c|c|}
\hline
$h$     & $L^{\infty}(L^2)$ error      & order  & $L^2(H^1)$ error & order \\ \hline
5.0000E-01  &  9.30E-02  &        &  3.10E-02  &       \\
2.5000E-01  &  3.47E-02  &  1.42  &  1.94E-02  &  0.67 \\
1.2500E-01  &  1.10E-02  &  1.66  &  9.38E-03  &  1.05 \\
6.2500E-02  &  3.66E-03  &  1.58  &  4.63E-03  &  1.02 \\
\hline
\end{tabular}
\caption{Test 1: Spatial errors and convergence rates with $\epsilon = 0.1$, $\delta = 25$, $\tau = 5 \times 10^{-5}$ and $M = 10^4$.}
\label{tab:ex1:h-rate}
\end{center}
\end{table}

\subsection{Test 2}
In this test, we take the initial function to be  
\begin{align*}
u_0(x) = \tanh \left( \frac{d_0(x)}{\sqrt{2} \epsilon} \right),
\end{align*}
where $d_0(x)$ represents the signed distance function to the ellipse 
\begin{align*}
\frac{x_1^2}{0.36} + \frac{x_2^2}{0.04} = 1.
\end{align*}

First, we investigate the evolution of the zero-level set with respect to the noise intensity 
$\delta$ with fixed $\epsilon = 0.01$. Figure~\ref{fig:ex2:deltatest_interface} plots snapshots 
at several time points of the zero-level set of $\bar{u}_h$ for $\delta = 1, 5, 10$, where 
\[
\bar{u}_h = \frac{1}{M} \sum_{i=1}^{M} u_h (\omega_i). 
\]
We observe that when the noise is relatively small ($\delta = 1$), the zero-level set is close to the 
deterministic interface (for $\delta = 0$). However, for relatively large noises ($\delta = 5, 10$), the zero-level 
sets rotate and evolve faster. 

\begin{figure}[htb]
\centering
\includegraphics[height=1.8in,width=2.4in]{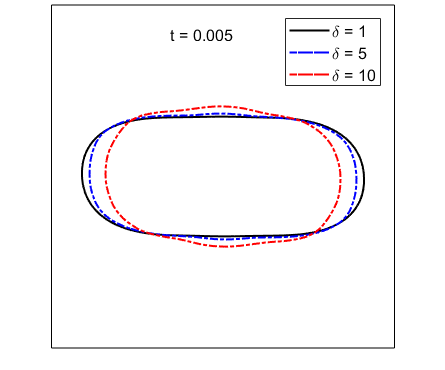}
\includegraphics[height=1.8in,width=2.4in]{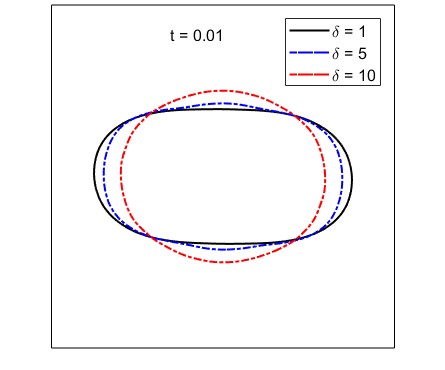}
\includegraphics[height=1.8in,width=2.4in]{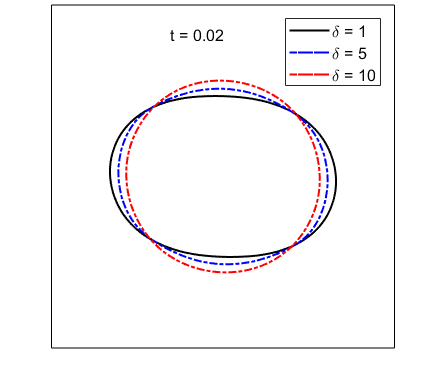}
\includegraphics[height=1.8in,width=2.4in]{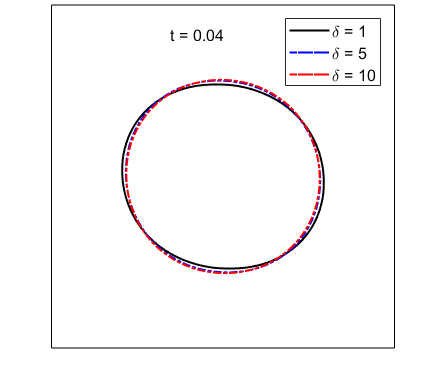}
\caption{Test 2: Snapshots of the zero-level set of $\bar{u}_h $ at several time points with $\delta=1, 5, 10$ and $\epsilon = 0.01$.}
\label{fig:ex2:deltatest_interface}
\end{figure}

Next, we fix $\delta = 1$ and study the influence of the parameter $\epsilon$ on the evolution of the numerical interfaces. In Figure~\ref{fig:ex2:epstest_interface}, snapshots at four fixed time points of the zero-level set of $\bar{u}_h$ are depicted for three different 
$\epsilon = 0.01, 0.015, 0.04$. Numerical results suggest the convergence of the numerical interface to the stochastic Hele-Shaw flow as $\epsilon \to 0$ at each of four time points. In addition, the numerical interface evolves faster in time for larger $\epsilon$. 

\begin{figure}[htb]
\centering
\includegraphics[height=1.8in,width=2.4in]{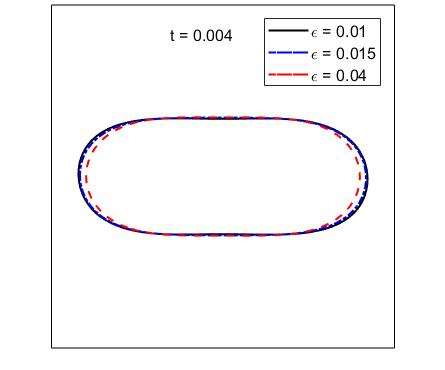}
\includegraphics[height=1.8in,width=2.4in]{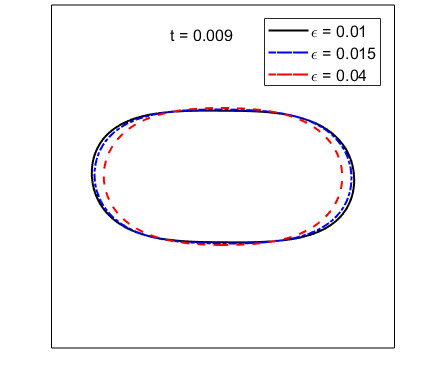}
\includegraphics[height=1.8in,width=2.4in]{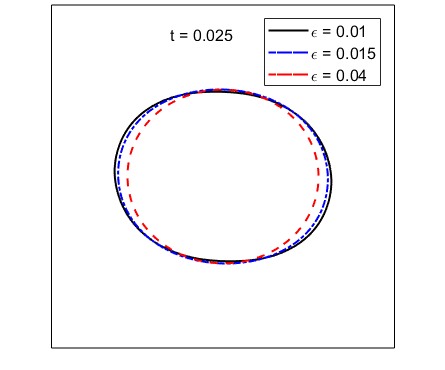}
\includegraphics[height=1.8in,width=2.4in]{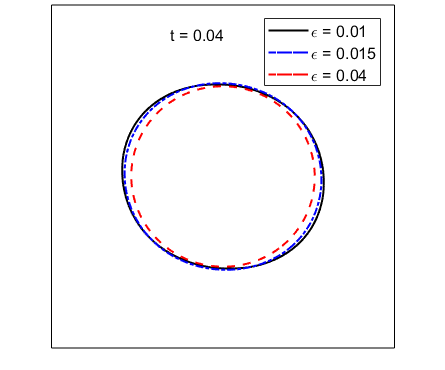}
\caption{Test 2: Snapshots of the zero-level set of $\bar{u}_h $ at several time points with $\epsilon =0.01, 0.015, 0.04$ and $\delta = 1$.}
\label{fig:ex2:epstest_interface}
\end{figure}

Notice that in Figure~\ref{fig:ex2:deltatest_interface}--\ref{fig:ex2:epstest_interface}, we only plot the evolutions 
on the subdomain $[-0.6,0.6]^2$ for a better resolution. 

In Figure~\ref{fig:ex2:energy}, we plot the change of the expected value of the discrete energy
\[
\Eb{J(u^n_h)} \approx \frac{1}{M} \sum_{i=1}^m J(u^n_h(\omega_i))
\]
in time with fixed $\epsilon = 0.01$. 
It indicates that the decay property still holds for $\delta = 1, 5, 10$. 

\begin{figure}[htb]
\centering
\includegraphics[scale=0.5]{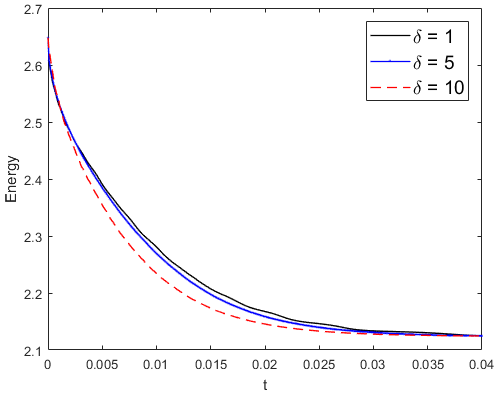}
\caption{Test 2: Decay of the expectation of numerical energy with $\epsilon = 0.01$.}
\label{fig:ex2:energy}
\end{figure}

\subsection{Test 3}

In this test, we consider the case with 
\begin{align*}
u_0(x) = \tanh \left( \frac{d_0(x)}{\sqrt{2} \epsilon} \right),
\end{align*}
where $d_0(x) = \min\{d_1(x), d_2(x)\} $, and $d_1(x)$ and $d_2(x)$ denotes respectively 
the signed distance function to the ellipses 
\begin{align*}
\frac{(x_1+0.2)^2}{0.15^2} + \frac{x_2^2}{0.45^2} = 1 \quad \text{and} \quad \frac{(x_1-0.2)^2}{0.15^2} + \frac{x_2^2}{0.45^2} = 1. 
\end{align*}

In Figure~\ref{fig:ex3:deltatest_interface}, 
we depict snapshots at several time points of the zero-level set 
of $\bar{u}_h$ for $\delta = 1, 5, 10$ with fixed parameter $\epsilon = 0.01$. 
For all cases, the two separated zero-level sets eventually merge and evolve to a circular shape. For larger noise intensity ($\delta = 5, 10$), the two interfaces merge faster and develop two concentric interfaces where the outer interface evolves to a circular shape and the inner interface shrinks and eventually vanishes. 

\begin{figure}[htb]
\centering
\includegraphics[height=1.8in,width=2.4in]{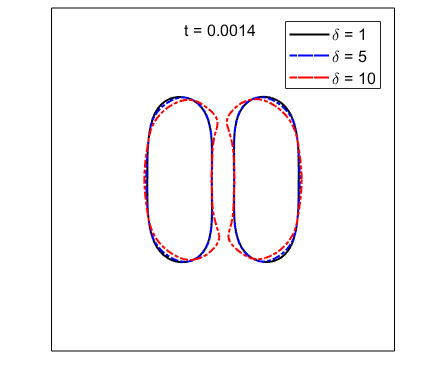}
\includegraphics[height=1.8in,width=2.4in]{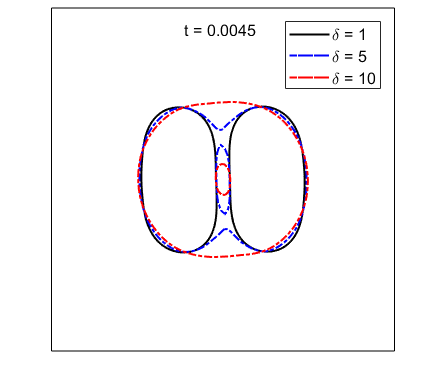}
\includegraphics[height=1.8in,width=2.4in]{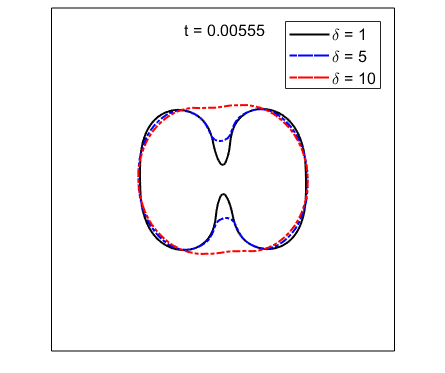}
\includegraphics[height=1.8in,width=2.4in]{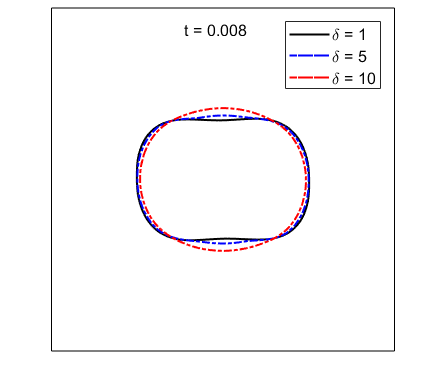}
\caption{Test 3: Snapshots of the zero-level set of $\bar{u}_h $ at several time points with $\delta=1, 5, 10$ and $\epsilon = 0.01$.}
\label{fig:ex3:deltatest_interface}
\end{figure}

Next, we plot a few snapshots of the zero-level set of $\bar{u}_h$ for $\epsilon = 0.01, 0.015, 0.04$ with fixed $\delta = 1$ in Figure~\ref{fig:ex3:epstest_interface}.  Again, the numerical interface evolves faster in time for larger $\epsilon$, and the numerical interfaces stay close for $\epsilon  = 0.01$ and $\epsilon = 0.015$. 
\begin{figure}[htb]
\centering
\includegraphics[height=1.8in,width=2.4in]{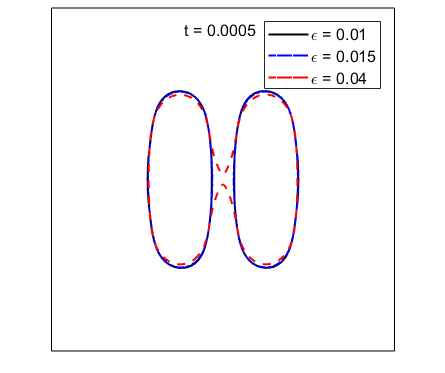}
\includegraphics[height=1.8in,width=2.4in]{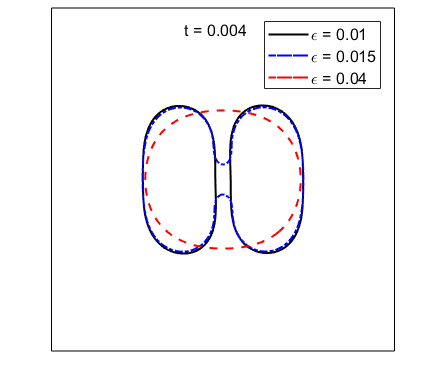}
\includegraphics[height=1.8in,width=2.4in]{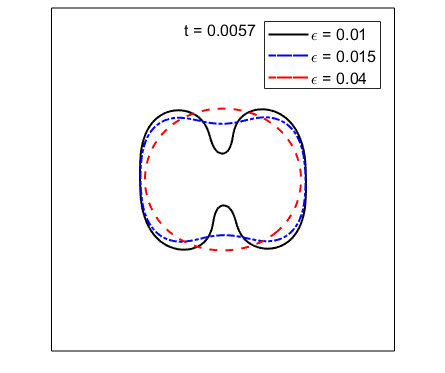}
\includegraphics[height=1.8in,width=2.4in]{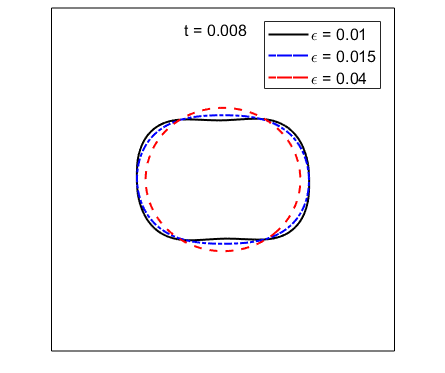}
\caption{Test 3: Snapshots of the zero-level set of $\bar{u}_h $ at several time points with $\epsilon =0.01, 0.015, 0.04$ and $\delta = 1$.}
\label{fig:ex3:epstest_interface}
\end{figure}

The decay of the expected value of the discrete energy is shown in Figure~\ref{fig:ex3:energy}, 
where we consider three noise intensity levels $\delta = 1, 5, 10$ with fixed $\epsilon = 0.01$. 
\begin{figure}[htb]
\centering
\includegraphics[scale=0.5]{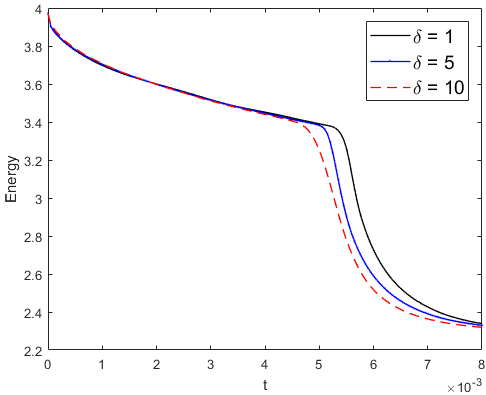}
\caption{Test 3: Decay of the expectation of numerical energy with $\epsilon = 0.01$.}
\label{fig:ex3:energy}
\end{figure}


\end{document}